\documentclass[12pt,reqno,a4paper]{amsart}%
\usepackage{amsfonts}
\usepackage{amsmath}
\usepackage{amssymb}
\usepackage{graphicx}%
\newtheorem{theorem}{Theorem}[section]

\newtheorem{lemma}[theorem]{Lemma}
\theoremstyle{definition}
\newtheorem{definition}[theorem]{Definition}
\newtheorem{remark}[theorem]{Remark}
\numberwithin{equation}{section} 
\begin{document}
\title[Hankel determinant for starlike and convex functions ]{Hankel determinant for starlike and convex functions of order $\frac{\alpha
}{2}-1$}
\author{A.A. AMOURAH}
\address{A.A. AMOURAH: Department of Mathematics, Faculty of Science and Technology,
Irbid National University, Irbid, Jordan.}
\email{alaammour@yahoo.com}
\author{ANAS ALJARAH}
\author{M. DARUS}
\address{ANAS ALJARAH and M. DARUS: School of Mathematical Sciences Faculty of Science
and Technology Universiti Kebangsaan Malaysia Bangi 43600 Selangor D. Ehsan,
Malaysia. }
\email{anasjrah@yahoo.com, maslina@ukm.edu.my.}

\begin{abstract}
The aim of this paper is to obtain an upper bound to the second Hankel
determinant $\left\vert a_{2}a_{4}-a_{3}^{2}\right\vert $ for starlike and
convex functions of order $\frac{\alpha}{2}-1$ $(1\leq\alpha\leq2)$.

\end{abstract}
\maketitle

\section{Introduction and preliminaries}

Let $\mathcal{A}$ denote the class of functions $f$ of the form:%
\begin{equation}
f(z)=z+\sum\limits_{n=2}^{\infty}a_{n}z^{n},\text{ \ \ \ \ } \label{1}%
\end{equation}
in the open unit disc $\mathbb{U}=\{z\in\mathbb{C}:$ $\left\vert z\right\vert
<1\}.$

The Hankel determinant of $f$ for $q\geq1$ and $n\geq1$ was defined by
Pommerenke(\cite{pom}, \cite{pom1}) as%

\begin{equation}
H_{q}(n)=\left\vert
\begin{array}
[c]{cccc}%
a_{n} & a_{n+1} & ... & a_{n+q-1}\\
a_{n+1} & ... & ... & ...\\
... & ... & ... & ...\\
a_{n+q-1} & ... & ... & a_{n+2q-2}%
\end{array}
\right\vert \label{10}%
\end{equation}

This determinant has been considered by many authors in the literature
\cite{tom}. For example Noor \cite{nor} determined the rate of growth of
$H_{q}(n)$ as $n\rightarrow\infty$ for functions given by (\ref{1}) with
bounded boundary. Ehrenborg \cite{eh} studied the Hankel determinant of
exponential polynomials. Janteng et al. discussed the Hankel determinant
problem for the classes of starlike functions with respect to symmetric points
and convex functions with respect to symmetric points in \cite{hal} and for
the functions whose derivative has a positive real part in \cite{hal1}.

Easily, one can observe that the Fekete and Szeg\"{o} functional is $H_{2}%
(1)$. Fekete and Szeg\"{o} \cite{fz} then further generalised the estimate
$\left\vert a_{3}-\mu a_{2}^{2}\right\vert $, where $\mu$ is real and $f\in
S$. For our discussion in this paper, we consider the Hankel determinant in
the case of $q=2$ and $n=2$:%
\[
\left\vert
\begin{array}
[c]{cc}%
a_{2} & a_{3}\\
a_{3} & a_{4}%
\end{array}
\right\vert .
\]
Let $P$ denote the class of functions%
\begin{equation}
P(z)=1+c_{1}z+c_{2}z^{2}+c_{3}z^{3}+\cdots=1+\sum\limits_{n=1}^{\infty}%
c_{n}z^{n}, \label{21}%
\end{equation}

which are analytic in $\mathbb{U}$ and satisfy $\operatorname{Re}\left\{
P(z)\right\}  $ $>0$ for any $z\in\mathbb{U}$.

In this paper, we seek sharp upper bound to the functional $\left\vert
a_{2}a_{4}-a_{3}^{2}\right\vert $ for the function $f$ belonging to the class
$S_{\alpha}^{\ast}$ and $C_{\alpha}$. The class $S_{\alpha}^{\ast}$ and
$C_{\alpha}$ are defined as follows.

\begin{definition}
\label{def123} Let $f$ be given by (\ref{1}). Then $f\in$ $S_{\alpha}^{\ast}$
if and only if
\begin{equation}
\operatorname{Re}\left\{  \frac{zf^{\prime}(z)}{f(z)}\right\}  >\frac{\alpha
}{2}-1,\text{ \ }(1\leq\alpha\leq2,\text{ }z\in\mathbb{U}). \label{77}%
\end{equation}

\end{definition}

\begin{definition}
\label{def1} Let $f$ be given by (\ref{1}). Then $f\in$ $C_{\alpha}$ if and
only if
\begin{equation}
\operatorname{Re}\left\{  1+\frac{zf^{\prime\prime}(z)}{f^{\prime}%
(z)}\right\}  >\frac{\alpha}{2}-1,\text{ \ }(1\leq\alpha\leq2,\text{ }%
z\in\mathbb{U}). \label{88}%
\end{equation}

\end{definition}

\bigskip\ To prove our main result in the next section, we shall require the
following two Lemmas:

\begin{lemma}
\label{lem1} (\cite{ch}, \cite{b})If $c\in P,$ then $\left\vert c_{n}%
\right\vert \leq2,$ for each $n\geq1$.
\end{lemma}

\begin{lemma}
\label{lem22}(\cite{z}, \cite{r}, \cite{r1})If $c\in P,$ then%
\[
2c_{2}=c_{1}^{2}+(4-c_{1}^{2})x,
\]%
\[
4c_{3}=c_{1}^{3}+2c_{1}(4-c_{1}^{2})x-c_{1}(4-c_{1}^{2})x^{2}+2(4-c_{1}%
^{2})(1-\left\vert x\right\vert ^{2})z,
\]

for some $x$ and $z$ satisfying $\left\vert x\right\vert \leq1$, $\left\vert
z\right\vert \leq1$ and $c_{1}\in\left[  0,2\right]  $.
\end{lemma}

We employ techniques similar to these used earlier by Amourah et al.
({\cite{c1}, \cite{c2}, \cite{c3}, \cite{c4}, \cite{c6}) and }Al-Hawary et al.
{\cite{c5}.}

\section{Main Result}

\begin{theorem}
\label{thm221} If $f(z)\in$ $S_{\alpha}^{\ast},$ then
\begin{equation}
\left\vert a_{2}a_{4}-a_{3}^{2}\right\vert \leq\frac{(2-\frac{\alpha}{2})^{2}%
}{3}\left[  4\left(  (2-\frac{\alpha}{2})^{2}-1\right)  +3\right]  ,\text{
}(1\leq\alpha\leq2,\text{ }z\in\mathbb{U}). \label{rr}%
\end{equation}

\end{theorem}

\begin{proof}
Since, $f\in S_{\alpha}^{\ast}$, by Definition \ref{def123} we have%
\begin{equation}
\frac{1-\frac{\alpha}{2}+\frac{zf^{\prime}(z)}{f(z)}}{(2-\frac{\alpha}{2}%
)}=P(z). \label{22}%
\end{equation}

Replacing $f(z)$ and $q(z)$ with their equivalent series expressions in
(\ref{22}), we have%
\[
\left(  1-\frac{\alpha}{2}+\frac{zf^{\prime}(z)}{f(z)}\right)  =(2-\frac
{\alpha}{2})P(z).
\]

Using the Binomial expansion in the left hand side of the above expression,
upon \ simplification, we obtain%
\begin{align}
&  (2-\frac{\alpha}{2})z+(2-\frac{\alpha}{2})\left(  c_{1}+a_{2}\right)
z^{2}+\left[  c_{2}+a_{2}c_{1}+a_{3}\right]  (2-\frac{\alpha}{2}%
)z^{3}\nonumber\\
&  +\left[  c_{3}+a_{2}c_{2}+a_{3}c_{1}+a_{4}\right]  (2-\frac{\alpha}%
{2})z^{4}+...\nonumber\\
&  =(2-\frac{\alpha}{2})z+(3-\frac{\alpha}{2})a_{2}z^{2}+(4-\frac{\alpha}%
{2})a_{3}z^{3}+(5-\frac{\alpha}{2})a_{4}z^{4}+\cdots. \label{112}%
\end{align}

On equating coefficients in (\ref{112}), we get%
\begin{align}
a_{2}  &  =(2-\frac{\alpha}{2})c_{1},\text{ }a_{3}=\frac{(2-\frac{\alpha}{2}%
)}{2}(c_{2}+(2-\frac{\alpha}{2})c_{1}^{2}),\nonumber\\
a_{4}  &  =\frac{(2-\frac{\alpha}{2})}{6}\left[  2c_{3}+3(2-\frac{\alpha}%
{2})c_{1}c_{2}+(2-\frac{\alpha}{2})^{2}c_{1}^{3}\right]  , \label{we}%
\end{align}

in the second Hankel functional%
\[
\left\vert a_{2}a_{4}-a_{3}^{2}\right\vert =\frac{(2-\frac{\alpha}{2})^{2}%
}{12}\left\vert 4c_{1}c_{3}-3c_{2}^{2}-c_{1}^{4}(2-\frac{\alpha}{2}%
)^{2}\right\vert .
\]

Using Lemma \ref{lem22}, it gives%
\begin{align*}
&  \left\vert a_{2}a_{4}-a_{3}^{2}\right\vert \\
&  =(2-\frac{\alpha}{2})^{2}\left\vert
\begin{array}
[c]{c}%
\frac{(4-c_{1}^{2})c_{1}^{2}x}{24}-\frac{\left(  4(2-\frac{\alpha}{2}%
)^{2}-1\right)  c_{1}^{4}}{48}\\
+\frac{(4-c_{1}^{2})(1-\left\vert x\right\vert ^{2})c_{1}z}{6}-\frac
{(4-c_{1}^{2})x^{2}(12+c_{1}^{2})}{48}%
\end{array}
\right\vert .
\end{align*}

Assume that $\delta=\left\vert x\right\vert \leq1,$ $c_{1}=c$ and $c$
$\in\lbrack0,2]$, using triangular inequality and $|z|\leq1$, we have%
\begin{align}
&  \left\vert a_{2}a_{4}-a_{3}^{2}\right\vert \nonumber\\
&  \leq(2-\frac{\alpha}{2})^{2}\left\{
\begin{array}
[c]{c}%
\frac{(4-c^{2})c^{2}\delta}{24}+\frac{\left(  4(2-\frac{\alpha}{2}%
)^{2}-1\right)  c^{4}}{48}\\
+\frac{(4-c^{2})(1-\delta^{2})c}{6}+\frac{(4-c^{2})\delta^{2}(12+c^{2})}{48}%
\end{array}
\right\} \\
&  =\frac{(2-\frac{\alpha}{2})^{2}}{48}\left\{
\begin{array}
[c]{c}%
\left(  4(2-\frac{\alpha}{2})^{2}-1\right)  c^{4}+8(4-c^{2})c\\
+2(4-c^{2})c^{2}\delta+(c-6)(c-2)(4-c^{2})\delta^{2}%
\end{array}
\right\} \\
&  =F(c,\delta). \label{y}%
\end{align}

We next maximize the function $F(c,\delta)$ on the closed region $\left[
0,2\right]  \times\left[  0,1\right]  .$ Differentiating $F(c,\delta)$ in
(\ref{y}) partially with respect to $\delta,$ we get%
\begin{align*}
\frac{\partial F}{\partial\delta}  &  =\frac{(2-\frac{\alpha}{2})^{2}}%
{48}\left[  2(4-c^{2})c^{2}+2(c-6)(c-2)(4-c^{2})\delta\right] \\
&  =\frac{(2-\frac{\alpha}{2})^{2}}{24}\left[  c^{2}+(c-6)(c-2)\delta\right]
(4-c^{2})
\end{align*}

We have $\frac{\partial F}{\partial\delta}>0$. Thus $F(c,\delta)$ cannot have
a maximum in the interior of the closed square $\left[  0,2\right]
\times\left[  0,1\right]  $. Moreover, for fixed $c\in\left[  0,2\right]  $%
\[
\max_{0\leq\delta\leq1}F(c,\delta)=F(c,1)=G(c).
\]%
\begin{equation}
G(c)=\frac{(2-\frac{\alpha}{2})^{2}}{48}\left[  4\left(  (2-\frac{\alpha}%
{2})^{2}-1\right)  c^{4}+48\right]  . \label{x1}%
\end{equation}%
\begin{equation}
G^{\prime}(c)=\frac{(2-\frac{\alpha}{2})^{2}}{48}\left[  16\left(
(2-\frac{\alpha}{2})^{2}-1\right)  c^{3}\right]  . \label{x2}%
\end{equation}

From the expression (\ref{x2}), we observe that $G^{\prime}(c)\geq0$ for all
values of $0\leq c\leq2$ and $1\leq\alpha\leq2.$ Therefore, $G(c)$ is a
monotonically increasing function of $c$ in the interval $\left[  0,2\right]
$ so that its maximum value occurs at $c=2$. From (\ref{x1}), we obtain%
\begin{equation}
\max_{0\leq c\leq2}G(2)=\frac{(2-\frac{\alpha}{2})^{2}}{3}\left[  4\left(
(2-\frac{\alpha}{2})^{2}-1\right)  +3\right]  . \label{x3}%
\end{equation}

From the expressions (\ref{y}) and (\ref{x3}), we obtain
\begin{equation}
\left\vert a_{2}a_{4}-a_{3}^{2}\right\vert \leq\frac{(2-\frac{\alpha}{2})^{2}%
}{3}\left[  4\left(  (2-\frac{\alpha}{2})^{2}-1\right)  +3\right]  .
\label{317}%
\end{equation}

This completes the proof of our theorem \ref{thm221}.\ \ \ \ \ \ \ \ \ \ \ \ \ \ \ \ \ \ \ \ \ \ \ \ \ \ \ \ \ \ \ \ \ \ \ \ \ \ \ \ \ \ \ \ \ \ \ \ \ \ \ \ \ \ \ \ \ \ \ \ \ \ \ \ \ \ \ \ \ \ 
\end{proof}

In particular, considering $\alpha=2$ in Theorem \ref{thm221}, we have the
following result.

\begin{remark}
\bigskip\label{rem223} If $f(z)\in$ $S_{2}^{\ast},$ then%
\[
\left\vert a_{2}a_{4}-a_{3}^{2}\right\vert \leq1.
\]

\end{remark}

\bigskip This inequality is sharp and coincides with that of Janteng, Halim
and Darus \cite{bs}.

\begin{theorem}
\label{thm2222} If $f(z)\in$ $C_{\alpha},$ then
\begin{equation}
\left\vert a_{2}a_{4}-a_{3}^{2}\right\vert \leq\frac{(2-\frac{\alpha}{2})^{2}%
}{144}\left[  \frac{17(2-\frac{\alpha}{2})^{2}+2(2-\frac{\alpha}{2}%
)+17}{1+(2-\frac{\alpha}{2})^{2}}\right]  ,\text{ }(1\leq\alpha\leq2,\text{
}z\in\mathbb{U}). \label{ss1}%
\end{equation}

\end{theorem}

\begin{proof}
Since, $f\in C_{\alpha}$, by Definition \ref{def1} we have%
\begin{equation}
\frac{2-\frac{\alpha}{2}+\frac{zf^{\prime\prime}(z)}{f^{\prime}(z)}}%
{(2-\frac{\alpha}{2})}=P(z). \label{ee3}%
\end{equation}

Replacing $f(z)$ and $q(z)$ with their equivalent series expressions in
(\ref{ee3}), we have%
\[
\left(  2-\frac{\alpha}{2}+\frac{zf^{\prime\prime}(z)}{f^{\prime}(z)}\right)
=(2-\frac{\alpha}{2})P(z).
\]

Using the Binomial expansion in the left hand side of the above expression,
upon \ simplification, we obtain%
\begin{align}
&  (2-\frac{\alpha}{2})c_{1}z+\left[  c_{2}+2a_{2}c_{1}\right]  (2-\frac
{\alpha}{2})z^{2}\label{1121}\\
&  +\left[  c_{3}+2a_{2}c_{2}+3a_{3}c_{1}\right]  (2-\frac{\alpha}{2}%
)z^{3}+...\nonumber\\
&  =2a_{2}z+6a_{3}z^{2}+12a_{4}z^{3}+\cdots.\nonumber
\end{align}

On equating coefficients in (\ref{1121}), we get%
\begin{align}
a_{2}  &  =\frac{(2-\frac{\alpha}{2})c_{1}}{2},\text{ }a_{3}=\frac
{(2-\frac{\alpha}{2})}{6}(c_{2}+(2-\frac{\alpha}{2})c_{1}^{2}),\label{wew}\\
a_{4}  &  =\frac{(2-\frac{\alpha}{2})}{12}\left[  c_{3}+\frac{3(2-\frac
{\alpha}{2})}{2}c_{1}c_{2}+\frac{(2-\frac{\alpha}{2})^{2}}{2}c_{1}^{3}\right]
,\nonumber
\end{align}

in the second Hankel functional%
\begin{align*}
&  \left\vert a_{2}a_{4}-a_{3}^{2}\right\vert \\
&  =\frac{(2-\frac{\alpha}{2})^{2}}{144}\left\vert 6c_{1}c_{3}-4c_{2}%
^{2}+(2-\frac{\alpha}{2})c_{1}^{2}c_{2}-(2-\frac{\alpha}{2})^{2}c_{1}%
^{4}\right\vert .
\end{align*}

Using Lemma \ref{lem22}, it gives%
\begin{align*}
&  \left\vert a_{2}a_{4}-a_{3}^{2}\right\vert \\
&  =\frac{(2-\frac{\alpha}{2})^{2}}{144}\left\vert
\begin{array}
[c]{c}%
\left(  \frac{1+(2-\frac{\alpha}{2})-2(2-\frac{\alpha}{2})^{2}}{2}\right)
c_{1}^{4}+\left(  \frac{2+(2-\frac{\alpha}{2})}{2}\right)  c_{1}^{2}%
(4-c_{1}^{2})x\\
-\frac{(8+c_{1}^{2})}{2}(4-c_{1}^{2})x^{2}+3c_{1}(4-c_{1}^{2})(1-\left\vert
x\right\vert ^{2})z
\end{array}
\right\vert .
\end{align*}

Assume that $\delta=\left\vert x\right\vert \leq1,$ $c_{1}=c$ and $c$
$\in\lbrack0,2]$, using triangular inequality and $|z|\leq1$, we have%
\begin{align}
\left\vert a_{2}a_{4}-a_{3}^{2}\right\vert  &  \leq\label{yy}\\
&  \frac{(2-\frac{\alpha}{2})^{2}}{144}\left\{
\begin{array}
[c]{c}%
\left(  \frac{1+(2-\frac{\alpha}{2})-2(2-\frac{\alpha}{2})^{2}}{2}\right)
c^{4}+3c(4-c^{2})+\\
\left(  \frac{2+(2-\frac{\alpha}{2})}{2}\right)  c^{2}(4-c^{2})\delta
+\frac{(c-2)(c-4)(4-c^{2})}{2}\delta^{2}%
\end{array}
\right\} \nonumber\\
&  =F(c,\delta).\nonumber
\end{align}

We next maximize the function $F(c,\delta)$ on the closed region $\left[
0,2\right]  \times\left[  0,1\right]  .$ Differentiating $F(c,\delta)$ in
(\ref{yy}) partially with respect to $\delta,$ we get%
\[
\frac{\partial F}{\partial\delta}=\frac{(2-\frac{\alpha}{2})^{2}}{144}\left[
\left(  \frac{2+(2-\frac{\alpha}{2})}{2}\right)  c^{2}(4-c^{2}%
)+(c-2)(c-4)(4-c^{2})\delta\right]  .
\]

We have $\frac{\partial F}{\partial\delta}>0$. Thus $F(c,\delta)$ cannot have
a maximum in the interior of the closed square $\left[  0,2\right]
\times\left[  0,1\right]  $. Moreover, for fixed $c\in\left[  0,2\right]  $%
\[
\max_{0\leq\delta\leq1}F(c,\delta)=F(c,1)=G(c).
\]%
\[
G(c)=\frac{(2-\frac{\alpha}{2})^{2}}{144}\left\{
\begin{array}
[c]{c}%
\left(  \frac{1+(2-\frac{\alpha}{2})-2(2-\frac{\alpha}{2})^{2}}{2}\right)
c^{4}+3c(4-c^{2})+\\
\left(  \frac{2+(2-\frac{\alpha}{2})}{2}\right)  c^{2}(4-c^{2})+\frac
{(c-2)(c-4)(4-c^{2})}{2}%
\end{array}
\right\}  ,
\]%
\begin{equation}
=\frac{(2-\frac{\alpha}{2})^{2}}{144}\left\{  -\left[  (2-\frac{\alpha}%
{2})^{2}+1\right]  c^{4}+2\left[  (2-\frac{\alpha}{2})+1\right]
c^{2}+16\right\}  . \label{z1}%
\end{equation}%
\begin{equation}
G^{\prime}(c)=\frac{(2-\frac{\alpha}{2})^{2}}{144}\left\{  -4\left[
(2-\frac{\alpha}{2})^{2}+1\right]  c^{3}+4\left[  (2-\frac{\alpha}%
{2})+1\right]  c\right\}  . \label{z2}%
\end{equation}%
\begin{equation}
G^{\prime\prime}(c)=\frac{(2-\frac{\alpha}{2})^{2}}{144}\left\{  -12\left[
(2-\frac{\alpha}{2})^{2}+1\right]  c^{2}+4\left[  (2-\frac{\alpha}%
{2})+1\right]  \right\}  . \label{z3}%
\end{equation}

For Optimum value of $G(c)$, consider $G^{\prime}(c)=0$. From (\ref{z2}), we
get%
\begin{equation}
c\left(  -\left[  (2-\frac{\alpha}{2})^{2}+1\right]  c^{2}+\left[
(2-\frac{\alpha}{2})+1\right]  \right)  =0. \label{z4}%
\end{equation}

We now discuss the following Cases.

Case 1) If $c=0$, then, from (\ref{z3}), we obtain%
\[
G^{\prime\prime}(c)=\frac{(2-\frac{\alpha}{2})^{2}}{36}\left[  (2-\frac
{\alpha}{2})+1\right]  >0,\text{ for }1\leq\alpha\leq2.
\]

From the second derivative test, $G(c)$ has minimum value at $c=0$.

Case 2) If $c\neq0$, then, from (\ref{z3}), we get%
\begin{equation}
c^{2}=\frac{1+(2-\frac{\alpha}{2})}{1+(2-\frac{\alpha}{2})^{2}}. \label{z5}%
\end{equation}

Using the value of $c^{2}$ given in (\ref{z5}) in (\ref{z3}), after
simplifying, we obtain%
\[
G^{\prime\prime}(c)=\frac{-60(2-\frac{\alpha}{2})^{2}}{144}\left[
1+(2-\frac{\alpha}{2})\right]  <0,\text{ for }1\leq\alpha\leq2.
\]

By the second derivative test, $G(c)$ has maximum value at $c$, where $c^{2}$
given in (\ref{z5}). Using the value of $c^{2}$ given by (\ref{z5}) in
(\ref{z1}), upon simpli cation, we obtain%
\begin{equation}
\max_{0\leq c\leq2}G(c)=\frac{(2-\frac{\alpha}{2})^{2}}{144}\left[
\frac{17(2-\frac{\alpha}{2})^{2}+2(2-\frac{\alpha}{2})+17}{1+(2-\frac{\alpha
}{2})^{2}}\right]  . \label{z6}%
\end{equation}

Considering, the maximum value of $G(c)$ at $c$, where $c^{2}$ is given by
(\ref{z5}) , from (\ref{yy}) and (\ref{z6}), we obtain%
\begin{equation}
\left\vert a_{2}a_{4}-a_{3}^{2}\right\vert \leq\frac{(2-\frac{\alpha}{2})^{2}%
}{144}\left[  \frac{17(2-\frac{\alpha}{2})^{2}+2(2-\frac{\alpha}{2}%
)+17}{1+(2-\frac{\alpha}{2})^{2}}\right]  . \label{z7}%
\end{equation}

This completes the proof of our Theorem \ref{thm2222}.\ \ \ \ \ \ \ \ \ \ \ \ \ \ \ 
\end{proof}

\begin{remark}
\bigskip\label{rem2233} If $f(z)\in$ $C_{2},$ then%
\[
\left\vert a_{2}a_{4}-a_{3}^{2}\right\vert \leq\frac{1}{8}.
\]

This inequality is sharp and coincides with that of Janteng, Halim and Darus
\cite{bs}.
\end{remark}


\begin{thebibliography}{99}                                                                                               %


\bibitem {ch}C. Pommerenke \& G. Jensen, \emph{Univalent functions}, Vol. 25.
G\"{o}ttingen: Vandenhoeck und Ruprecht, (1975).

\bibitem {b}B. Simon, \emph{Orthogonal polynomials on the unit circle}, Part
1, volume 54 of American Mathematical Society Colloquium Publications."
American Mathematical Society, Providence, RI (2005).

\bibitem {z}U. Grenander \& G. Szeg\"{o}, \emph{Toeplitz Forms and Their
Applications}, 2nd edn (New York: Chelsea), (1984).

\bibitem {r}R.J. Libera \& E. J. Z\l otkiewicz, \emph{Early coefficients of
the inverse of a regular convex function}, Proceedings of the American
Mathematical Society 85. 2(1982), 225-230.

\bibitem {r1}R.J. Libera \& E. J. Z\l otkiewicz, \emph{Coefficient bounds for
the inverse of a function with derivative in }$P$, Proceedings of the American
Mathematical Society 87. 2(1983), 251-257.

\bibitem {pom}C. Pommerenke, \emph{On the Hankel determinants of univalent
functions}, Mathematika 14. 1(1967), 108-112.

\bibitem {pom1}C. Pommerenke, \emph{On the coefficients and Hankel
determinants of univalent functions}, Journal of the London Mathematical
Society 1. 1(1966), 111-122.

\bibitem {tom}J. W. Noonan \& D. K. Thomas, \emph{On the second Hankel
determinant of areally mean }$p-$\emph{valent functions}, Transactions of the
American Mathematical Society 223 (1976), 337-346.

\bibitem {nor}K. I. Noor, \emph{Hankel determinant problem for the class of
functions with bounded boundary rotation}, Revue Roumaine de Math\'{e}matiques
Pures et Appliqu\'{e}es 28. 8(1983), 731-739.

\bibitem {eh}R. Ehrenborg, \emph{The Hankel determinant of exponential
polynomials}, American Mathematical Monthly (2000), 557-560.

\bibitem {hal}A. Janteng, S. A. Halim \& M. Darus,\emph{ Hankel determinant
for functions starlike and convex with respect to symmetric points}, Journal
of Quality Measurement and Analysis 2. 1(2006), 37-43.

\bibitem {hal1}A. Janteng, S. A. Halim \& M. Darus, \emph{Coefficient
inequality for a function whose derivative has a positive real part}, Journal
of Inequalities in Pure and Applied Mathematics 7. 2(2006), 50.

\bibitem {fz}M. Fekete and G. Szeg\"{o}, \emph{Eine Bemerkung \"{u}ber
ungerade schlichte Funktionen}, Journal of the London Mathematical Society 1.
2(1933), 85-89.

\bibitem {bs}A. Janteng, S. A. Halim \& M. Darus. \emph{Hankel determinant for
starlike and convex functions}. Int. J. Math. Anal, 1. 13(2007), 619-625.

\bibitem {c1}A. A. Amourah, F. Yousef,T. Al-Hawary \& M. Darus. \emph{On
}$H_{3}(p)$\emph{ Hankel determinant for certain subclass of p-valent
functions}. Ital. J. Pure Appl. Math, 37 (2017), 611-618.

\bibitem {c2}A. A. Amourah, F. Yousef,T. Al-Hawary \& M. Darus. \emph{On a
Class of }$p-$\emph{Valent non-Bazilevi\v{c} Functions of Order }$\mu+i\beta$.
International Journal of Mathematical Analysis 10, 15(2016), 701-710.

\bibitem {c3}A. A. Amourah. \emph{Faber polynomial coefficient estimates for a
class of analytic bi-univalent functions}. arXiv preprint arXiv:1810.07018 (2018).

\bibitem {c4}A. A. Amourah and Mohamed Illafe. \emph{A Comprehensive subclass
of analytic and bi-univalent functions associated with subordination}. arXiv
preprint arXiv:1811.07731 (2018).

\bibitem {c5}T. Al-Hawary, A. A. Amourah and M. Darus. \emph{Differential
sandwich theorems for }$p-$\emph{valent functions associated with two
generalized differential operator and integral operator}. International
Information Institute (Tokyo). Information 17, 8(2014), 3559.

\bibitem {c6}A. A. Amourah, A. G. Alamoush and M. Darus. \emph{On generalised
p-valent non-bazilavec functions of order }$\alpha+i\beta$. arXiv preprint
arXiv:1902.09534 (2019).
\end{thebibliography}
\end{document}